\documentclass[3p, times]{elsarticle}
\usepackage{amsmath}
\usepackage{amssymb}
\usepackage{amsthm}

\newtheorem{theorem}{Theorem}
\newtheorem{remark}[theorem]{Remark}
\newtheorem{lem}[theorem]{Lemma}
\newtheorem{hyp}[theorem]{Assumption}

\newcommand{\dps}{\displaystyle}
\newcommand{\eps}{\varepsilon}
\newcommand{\RR}{\mathbb R}
\newcommand{\ZZ}{\mathbb Z}

\newcommand{\Aper}{A_\text{per}}
\newcommand{\approximations}{V^h}
\newcommand{\cell}[1]{Q_{#1}^h}
\newcommand{\cellindices}{\{0,\ldots,N-1\}^d}
\newcommand{\DFT}{\operatorname{DFT}}
\renewcommand{\div}{\operatorname{div}}
\newcommand{\freqindiceseven}{\{-M+1,\ldots,M\}^d}
\newcommand{\freqindicesodd}{\{-M,\ldots,M\}^d}
\newcommand{\green}{\Gamma}
\newcommand{\polarization}{\tau}
\newcommand{\polarizations}{V}
\newcommand{\sumvoxels}[1]{\sum_{#1_1=0}^{N-1}\cdots\sum_{#1_d=0}^{N-1}}
\newcommand{\test}{\sigma}
\newcommand{\trial}{\tau}
\newcommand{\unitcell}{Q}
\newcommand{\unitcellzero}{{Q_0}}
\newcommand{\Izero}{{\cal I}_0}
\newcommand{\I}{{\cal I}}

\journal{Parallel Computing}

\begin{document}

\begin{frontmatter}

\title{Periodic homogenization using the Lippmann--Schwinger formalism}

\author[adr:1]{S\'ebastien Brisard}
\ead{sebastien.brisard@ifsttar.fr}
\ead[url]{http://navier.enpc.fr/BRISARD-Sebastien}

\author[adr:1,adr:2]{Fr\'ed\'eric Legoll\corref{cor:1}}
\ead{legoll@lami.enpc.fr}
\ead[url]{http://navier.enpc.fr/LEGOLL-Frederic}

\cortext[cor:1]{Corresponding author}
\address[adr:1]{Universit\'e Paris-Est, Laboratoire Navier (UMR 8205), CNRS, ENPC, IFSTTAR, 77455 Marne-la-Vall\'ee, France}
\address[adr:2]{INRIA Rocquencourt, MATHERIALS research-team, 78153 Le Chesnay Cedex, France}

\begin{abstract}
When homogenizing elliptic partial differential equations, the so-called corrector problem is pivotal to compute the macroscale effective coefficients from the microscale information. To solve this corrector problem in the periodic setting, Moulinec and Suquet introduced in the mid-nineties a numerical strategy based on the reformulation of that problem as an integral equation (known as the Lippmann--Schwinger equation), which is then suitably discretized. This results in an iterative, matrix-free method, which is of particular interest for complex microstructures. Since the seminal work of Moulinec and Suquet, several variants of their scheme have been proposed.

The aim of this contribution is twofold. First, we provide an overview of these methods, recast in the language of the applied mathematics community. These methods are presented as asymptotically consistent Galerkin discretizations of the Lippmann--Schwinger equation. The bilinear form arising in the weak form of this integral equation is indeed the sum of a local and a non-local term. We show that most of the variants proposed in the literature correspond to alternative approximations of this non-local term. Second, we propose a mathematical analysis of the discretized problem. In particular, we prove under mild hypotheses the convergence of these numerical schemes with respect to the grid-size. We also provide a priori error estimates on the solution.

The article closes on a three-dimensional numerical application within the framework of linear elasticity.
\end{abstract}


\end{frontmatter}

\section{Introduction}
\label{sec:intro}

When homogenizing elliptic partial differential equations (PDEs), the so-called corrector problem is pivotal to compute the macroscale effective coefficients from the microscale information. In this article, we are concerned with a numerical strategy, alternative to finite element methods, to solve that corrector problem in the periodic setting. That alternative strategy was proposed in the computational mechanics community in the 1990s, and our aim here is to review that strategy and its variant in the language of the applied mathematics community. 

\medskip

The highly oscillatory problem we start from reads (see Section~\ref{sec:base_hom} for more details)
\begin{equation}
\label{eq:pb_eps_intro}
-\div\left[\Aper\left( \frac{x}{\eps} \right) \nabla u^\eps \right] = f\quad\text{in $\Omega$},
\quad
u^\eps = 0\quad\text{on $\partial \Omega$},
\end{equation}
on a bounded domain $\Omega \subset \RR^d$, where the matrix $\Aper$ is assumed to be bounded, coercive, symmetric and periodic. Problem~\eqref{eq:pb_eps_intro} typically models the conduction of heat in a heterogeneous material (in that case, $u^\eps$ is the temperature and $\dps\Aper\left( \frac{x}{\eps} \right)$ is the heat conduction matrix), or models the Darcy law in subsurface flows (in that case, $u^\eps$ is the pressure and $\dps A_{\rm per}\left( \frac{x}{\eps} \right)$ is the materials permeability). We refer to~\cite[Chapter 1]{hornung} for a short non technical overview of related problems. In any case, the material is supposed to be heterogeneous, with physical or mechanical properties varying at the small scale $\eps$ (see Fig.~\ref{fig:1} below for a particular illustration in composite materials). Consequently, the solution $u^\eps$ to~\eqref{eq:pb_eps_intro} also varies at scale $\eps$. It is therefore challenging to numerically approximate. If one were to use a finite element discretization of~\eqref{eq:pb_eps_intro}, one would have to use elements with size of the order of $\eps$, which would be very expensive.

\medskip

Under the above assumptions, problem~\eqref{eq:pb_eps_intro} admits a homogenized limit when $\eps$ goes to 0, which is given by
\begin{equation}
  -\div\left[A^\star \nabla u^\star \right] = f \quad \text{in $\Omega$},
  \quad
  u^\star = 0\quad\text{ on $\partial \Omega$},
  \label{PB:homo_intro}
\end{equation}
where the homogenized matrix $A^\star$ is constant, and can be directly computed (see~\eqref{eq:def_Astar_corr} below) from the so-called corrector functions $w_p$, that solve, for any $p \in \RR^d$, the corrector problem
\begin{equation}
-\div\left[\Aper(p + \nabla w_p) \right] = 0 \quad \text{in $\RR^d$},
\qquad
w_p\text{ is $\ZZ^d$-periodic.}
\label{PB:correcteur_intro}
\end{equation}
Note that, in~\eqref{PB:homo_intro}, the matrix $A^\star$ is constant. Therefore problem~\eqref{PB:homo_intro} is easy to solve. The practical interest of the approach is obvious. No small scale $\eps$ is present in the homogenized problem~\eqref{PB:homo_intro}. At the price of only computing $d$ periodic problems~\eqref{PB:correcteur_intro} (i.e. as many problems as dimensions in the ambient space), the solution to problem~\eqref{eq:pb_eps_intro} can be efficiently approached for $\eps$ small. A direct attack of problem~\eqref{eq:pb_eps_intro} would require taking a meshsize smaller than $\varepsilon$. The difficulty has been circumvented. 

\medskip

In the sequel of this article, we focus on how to discretize the corrector problem~\eqref{PB:correcteur_intro}. An obvious approach is to use a finite element discretization. We describe here an alternative approach, introduced in the computational mechanics community by Moulinec and Suquet in~\cite{suquet1,suquet2}, which amounts to recasting~\eqref{PB:correcteur_intro} as an integral equation, the so-called Lippmann--Schwinger equation (see~\eqref{eq:4} below), and taking advantage of the periodic boundary conditions by using Fourier representations. We refer to Section~\ref{sec:base_LS} for more details.

Following the seminal work of Moulinec and Suquet~\cite{suquet1,suquet2}, several variants of the approach have been proposed, both within the framework of thermal conductivity (which is our framework here) and linear elasticity. In~\cite{EYRE1999}, Eyre and Milton introduced an \emph{accelerated} iterative scheme, combined with a multigrid approach. In~\cite{MICH2001}, Michel, Moulinec and Suquet analyzed the convergence of their seminal scheme and of accelerated iterative schemes. They also introduced a modified iterative scheme, based on an augmented Lagrangian approach, which allows for infinite contrast within the microstructure. More recently, Monchiet and Bonnet have proposed in~\cite{MONC2012} yet another iterative scheme, which combines (within the framework of linear elasticity) the Green operator for strains and the Green operator for stresses.

In the above-cited works, an ad-hoc iterative scheme is first proposed to solve the Lippmann--Schwinger equation (e.g. a fixed point algorithm). Next, the equation is discretized in space. However, it has been recognized that it may be more interesting to first discretize the Lippmann--Schwinger equation, thereby obtaining a linear system, and to next solve this linear system using standard iterative linear solvers, such as conjugate gradient~\cite{BRIS2010A, ZEMA2010, GELE2013} or SYMMLQ~\cite{BRIS2012A} solvers. 

We note that the extension of this approach (which is based on the reformulation of the corrector problem~\eqref{PB:correcteur_intro} as the integral equation~\eqref{eq:4}) to more complex behaviors has been also studied in the literature. We refer to~\cite{MICH2001, VINO2008, GELE2013}, \cite{MICH2001, LEBE2011} and~\cite{LEBE2012} for examples in non-linear elasticity, plasticity and elasto-viscoplasticity, respectively. 

\medskip

While a fair amount of work has been devoted to studying and improving the convergence of the various iterative schemes (for a \emph{fixed} value of the spatial discretization parameter), convergence with respect to the discretization parameter itself (regardless of the actual iterative solver) has hardly been approached. In the recent paper~\cite{BRIS2012A}, one of the authors has shown that all the schemes mentioned above can be seen in a unified framework as asymptotically consistent Galerkin discretizations of the Lippmann--Schwinger equation. Convergence analysis can therefore readily be studied, using standard tools such as C\'ea lemma (if the problem is coercive) or, more generally, the Banach-Necas-Babuska (BNB) theorem. The work~\cite{BRIS2012A} helps to draw a clear separation between (i) spatial discretization of the continuous Lippmann--Schwinger equation on the one hand, and (ii) resolution of the resulting linear system on the other hand, by means of standard Krylov subspace methods or ad-hoc solvers (such as the fixed-point scheme of~\cite{suquet1, suquet2} or the accelerated iterations of~\cite{EYRE1999}). 

The Lippmann--Schwinger equation is written in terms of a so-called continuous Green operator, denoted $\green_0$ in~\eqref{eq:4} below. Several asymptotically consistent approximations of $\green_0$, called discrete Green operators, can be introduced, which are much more amenable to practical evaluation and implementation. As explained in~\cite{BRIS2012A}, all the above mentioned schemes can be seen as the combination of a specific discrete Green operator and a specific linear solver. This classification opens the way to the implementation of the various approaches within an object-oriented framework. In our code, various types of physics (including scalar problems, such as thermal conductivity or electrical response, vectorial problems such as linear elasticity, \dots, in any spatial dimension) as well as various discretizations of the continuous Green operator $\green_0$ can readily be combined with various linear solvers.

\bigskip

The aim of this contribution is to review the numerical schemes discussed above within the Galerkin framework and the terminology introduced in~\cite{BRIS2012A}. We also extended the mathematical analysis that was carried out in~\cite{BRIS2012A} to more general cases. The analysis in~\cite{BRIS2012A} focuses on isotropic linear elastic materials. Here, we do not make any assumption concerning isotropy. Furthermore, we provide some error bounds on the approach (see Theorem~\ref{theo:error} below) in terms of the spatial discretization parameter.

A natural question is the comparison of the Lippmann--Schwinger approach with an approach that directly attacks~\eqref{PB:correcteur_intro}, e.g. using a finite element discretization. To the best of our knowledge, such a comparison has not been carried out. Note that, to be fair, the present FFT-based approach should be compared with a \emph{specific} implementation of the finite element method taking advantage of the grid-like nature of the mesh. More precisely, neither the geometry of the mesh (coordinates of nodes, connectivity, \dots), nor the global stiffness matrix need to be stored. This results in a matrix-free approach, which can be handled by iterative linear solvers. The so-called element-by-element (EbE) method~\cite{BARR1988, CARE1988, TERA1997, ARBE2008} would be a potential match to FFT-based methods. Such a comparison is out of the scope of this contribution.

\medskip

This article is organized as follows. In Section~\ref{sec:basics}, we begin with a brief introduction to periodic homogenization theory, before turning to the Lippmann-Schwinger formalism and the reformulation of the corrector equation. In Section~\ref{sec:disc}, we discuss the numerical discretization of the Lippmann--Schwinger equation. Section~\ref{sec:math} is devoted to the mathematical analysis of the continuous and discrete problems. We also establish there error bounds. In Section~\ref{sec:num}, we finally turn to a numerical illustration of the approach on a three-dimensional problem.

\section{Periodic homogenization and the Lippmann--Schwinger equation}
\label{sec:basics}

This section is devoted to first introducing the basic setting of periodic homogenization (see Section~\ref{sec:base_hom}). There is of course no novelty in such an introduction, the only purpose of which is the consistency of the contribution and the convenience of the reader not familiar with the theory. We refer to e.g. the monographs~\cite{blp,cd,jikov}  for more details on homogenization theory and to~\cite[Chapters 1 and 2]{allaire-book} for a pedagogic presentation. We next introduce the Lippmann--Schwinger formalism in Section~\ref{sec:base_LS}. 

Consider the highly oscillatory equation
\begin{equation}
\label{eq:pb_eps}
-\div\left[\Aper\left( \frac{x}{\eps} \right) \nabla u^\eps \right] = f\quad\text{in $\Omega$},
\quad
u^\eps = 0\quad\text{on $\partial \Omega$},
\end{equation}
on a bounded domain $\Omega \subset \RR^d$, with $f \in L^2(\Omega)$. In the above equation, $\Aper \in \left( L^\infty(\RR^d) \right)^{d \times d}$ is a matrix that we assume to be $\ZZ^d$ periodic, i.e.
\begin{equation*}
\forall k \in \ZZ^d, \quad \Aper(x+k) = \Aper(x) \quad \text{a.e. on $\RR^d$},
\end{equation*}
and uniformly coercive: there exists a constant $a_- > 0$ such that
\begin{equation}
\label{eq:hyp1}
\forall \xi \in \RR^d, \quad a_- |\xi|^2 \leq \xi^T \Aper(x) \xi\quad\text{a.e. on $\RR^d$}.
\end{equation}
Hereafter, we assume $\Aper$ to be a symmetric matrix. This is not a restriction with respect to the applications we have in mind, for which the PDE of interest is the Euler-Lagrange equation associated to some physical energy.

\subsection{Periodic homogenization theory}
\label{sec:base_hom}

Due to the periodicity assumption on $\Aper$, problem~\eqref{eq:pb_eps} admits an explicit homogenized limit when $\eps$ goes to 0. It is indeed well-known that the solution $u^\eps$ to~\eqref{eq:pb_eps} converges weakly in $H^1_0(\Omega)$ as $\eps \rightarrow 0$ to $u^\star$, solution to the homogenized equation
\begin{equation}
  -\div\left[A^\star \nabla u^\star \right] = f \quad \text{in $\Omega$},
  \quad
  u^\star = 0\quad\text{ on $\partial \Omega$}.
  \label{PB:homo}
\end{equation}
The homogenized matrix $A^\star$ is constant, and such that
\begin{equation}
\label{eq:def_Astar_corr}
\forall p \in \RR^d, \quad A^\star p = \int_{Q}\Aper(y) (p + \nabla w_p(y)) \,dy,
\qquad
Q=(0,1)^d,
\end{equation}
where, for any $p \in \RR^d$, $w_p$ is the unique (up to the addition of a constant) solution to the corrector problem
\begin{equation}
-\div\left[\Aper(p + \nabla w_p) \right] = 0 \quad \text{in $\RR^d$},
\qquad
w_p\text{ is $\ZZ^d$-periodic.}
\label{PB:correcteur}
\end{equation}

The above convergence result can be established using various techniques. One possible approach is the \emph{energy method} (i.e. the \emph{method of oscillating test functions}) introduced by Murat and Tartar (see~\cite{murat,tartar}), where one uses test functions in~\eqref{eq:pb_eps} of the form $\dps \varphi(x) + \eps \sum_{i=1}^d w_{e_i}\left(\frac{x}{\eps}\right)\, \frac{\partial \varphi}{\partial x_i}(x)$, for any $\varphi \in C^\infty_c(\Omega)$. Another possible approach is to use the notion of \emph{two-scale convergence} introduced by Nguetseng and developed by Allaire (see~\cite{allaire,nguetseng}). We refer to the bibliography for more details on both approaches.

The convergence of $u^\eps$ to $u^\star$ holds weakly in $H^1_0(\Omega)$, and strongly in $L^2(\Omega)$. The correctors $w_{e_i}$, $1 \leq i \leq d$, may also be used to {\em correct} $u^\star$ in order to identify the behavior of $u^\eps$ in the strong $H^1(\Omega)$ norm. We more precisely have
\begin{equation*}
\lim_{\eps \to 0} \left\|
u^{\varepsilon} - u^\star - \eps \sum_{i=1}^d w_{e_i}\left(\frac{\cdot}{\varepsilon}\right)\frac{\partial u^\star}{\partial x_i} \right\|_{H^1(\Omega)} = 0.
\end{equation*}
Under some regularity assumptions on $\Aper$ and $u^\star$, a sharp rate (in terms of $\eps$) can be established for the above convergence (see e.g.~\cite[p. 28]{jikov} or~\cite[Theorem 2.1]{amar}). Several other convergences on various products involving $\dps A_{\rm per}\left(\frac{x}{\eps} \right)$ and $u^\varepsilon$ also hold. All this is well documented.

\medskip

As pointed out above, in the sequel of this article, we focus on how to discretize the corrector problem~\eqref{PB:correcteur}. We focus here of an approach introduced in the computational mechanics community by Moulinec and Suquet in~\cite{suquet1,suquet2}, which amounts to recasting~\eqref{PB:correcteur} as an integral equation, the so-called Lippmann--Schwinger equation, and taking advantage of the periodic boundary conditions by using Fourier representations. 

\begin{remark}
\label{rem:random}
Note that the problem~\eqref{PB:correcteur} not only appears in periodic homogenization, but also in stochastic homogenization. In that case, the corrector equation, which plays the role of~\eqref{PB:correcteur}, is set on the {\em entire} space $\RR^d$. To approximate it, a standard procedure is to consider a {\em truncated} version of the stochastic corrector problem on the bounded domain $Q_N = (-N,N)^d$, complemented with e.g. periodic boundary conditions. The approximate corrector problem then reads
\begin{equation}
  \left\{
    \begin{aligned}
      &-\div\left[A(\cdot,\omega) \left(p + \nabla w_p^N(\cdot,\omega)\right) \right] = 0 \quad \text{in $\RR^d$, almost surely,}\\
      &w_p^N(\cdot,\omega)\text{ is $Q_N$-periodic, almost surely,}
    \end{aligned}
  \right.
  \label{PB:correcteur_N}
\end{equation}
which is of the same form as~\eqref{PB:correcteur}. We refer e.g. to~\cite{singapour} and the comprehensive bibliography therein for more details on stochastic homogenization. The approach we describe below can readily be extended to handle the case of~\eqref{PB:correcteur_N}.
\end{remark}

\subsection{The Lippmann--Schwinger equation}
\label{sec:base_LS}

We now recast the corrector problem~\eqref{PB:correcteur} as an integral equation, the so-called Lippmann--Schwinger equation. We first introduce the space $\polarizations = \left( L^2_{\rm per}(\RR^d) \right)^d$, that is
\begin{equation*}
\polarizations = \left\{ \polarization \colon \RR^d \to \RR^d, \quad \text{$\tau$ is $\ZZ^d$-periodic}, \quad \tau \in \left( L^2(\unitcell) \right)^d
\right\}, \qquad Q=(0,1)^d.
\end{equation*}

\subsubsection{Green function and the second-rank Green operator}

Let $A_0\in\RR^{d\times d}$ be a constant, symmetric, positive definite matrix, and $G_0$ be the Green function of the operator $L_0 = -\div\left(A_0\nabla \cdot\right)$ with periodic boundary conditions on $Q$, that is the function $G_0 : Q \to \RR$ (uniquely defined up to an additive constant) such that
\begin{equation*}
-\div \left(A_0\nabla G_0 \right)= \delta_0 - \frac{1}{|\unitcell|},
\qquad
G_0 \text{ is $\ZZ^d$-periodic}.
\end{equation*}
Note that the above right-hand side is of mean zero, and the above equation is well-posed. For any $\polarization \in \polarizations$, consider the problem
\begin{equation}
\label{eq:1}
-\div\left(A_0\nabla u+\polarization\right)=0\quad\text{in $\RR^d$},
\qquad
u\text{ is $\ZZ^d$-periodic}.
\end{equation}
We have
\begin{equation*}
\nabla u(x) = -\int_\unitcell \green_0(x-y) \, \polarization(y) \, dy,
\end{equation*}
where $\green_0 = - \nabla^2 G_0$ (see also~\citep{MILT2002}). In the sequel, we use the following notation:
\begin{equation*}
  \green_0 \ast \polarization(x) :=\int_\unitcell\green_0(x-y) \, \polarization(y) \, dy,
\end{equation*}
so that the solution to~\eqref{eq:1} satisfies $\nabla u=-\green_0\ast\polarization$, and $\Gamma_0$ can also be considered as an operator, the so-called second-rank Green operator associated to $A_0$.

\medskip

Taking advantage of the periodic boundary conditions in~\eqref{eq:1}, the above equation reads, in Fourier space (see~\citep{MILT2002}),
\begin{equation}
  \label{eq:2}
\green_0\ast\polarization(x)=\sum_{k\in\ZZ^d}\hat\green_0(k) \ \hat\polarization(k) \ \exp\left(2 i\pi k^Tx\right),
\end{equation}
where the Fourier coefficients $\hat\polarization(k)$ of $\polarization$ are given by
\begin{equation*}
  \hat\polarization(k)=\int_\unitcell\polarization(x)\exp\left(-2 i \pi k^Tx\right)dx
\end{equation*}
and likewise for $\hat{\green}_0(k)$. By definition of $\green_0$, the matrix $\hat{\green}_0(k)$ is given by
\begin{equation*}
  \hat\green_0(k)=
  \begin{cases}
    \left(k^TA_0k\right)^{-1}kk^T & \text{if $k\in\ZZ^d\setminus\{0\}$},\\
    0 & \text{ if $k=0$}.
  \end{cases}
\end{equation*}

\subsubsection{Lippmann--Schwinger equation}

We now go back to the corrector problem~\eqref{PB:correcteur} and recast it as an integral equation. To begin with, we select a constant, symmetric, positive definite matrix $A_0$ such that, appart from the domain (which can be empty) where $\Aper(x) = A_0$, we have that $\Aper(x) - A_0$ is bounded away from 0. More precisely, we assume that there exists $Q_0 \subset Q$ (with possibly $Q_0 = Q$) and $c_0 > 0$ such that
\begin{gather}
\label{eq:hyp2}
\begin{cases}
\Aper(x) = A_0 \quad \text{a.e. on $Q \setminus Q_0$},
\\
\forall \xi \in \RR^d, \quad c_0 |\xi|^2 \leq \left| \xi^T \left( \Aper(x) - A_0 \right) \xi \right| \quad \text{a.e. on $Q_0$.}
\end{cases}
\end{gather}
We introduce the auxiliary unknown $\polarization_p$ (called the polarization), defined as
\begin{equation}
  \label{eq:3}
  \polarization_p := \left(\Aper-A_0\right)\left(p+\nabla w_p\right)
\end{equation}
and infer from~\eqref{PB:correcteur} that
\begin{equation*}
-\div\left[A_0\left(p+\nabla w_p\right)+\polarization_p\right]=0.
\end{equation*}
Since $A_0 p$ is constant, we have
\begin{equation*}
-\div\left[A_0 \nabla w_p +\polarization_p\right]=0, \quad w_p \text{ is $\ZZ^d$-periodic},
\end{equation*}
and therefore, by definition of $\green_0$ (see~\eqref{eq:1}), we have $\nabla w_p=-\green_0\ast\polarization_p$. Using the definition~\eqref{eq:3} of $\polarization_p$, we get
\begin{gather}
  \label{eq:4}
\begin{cases}
  \left(\Aper-A_0\right)^{-1}\polarization_p+\green_0\ast\polarization_p =p \quad \text{a.e. on $Q_0$},\\
\polarization_p =0 \quad \text{a.e. on $Q \setminus Q_0$},\\
\polarization_p \in V.
\end{cases}
\end{gather}
In addition, for any $p \in \RR^d$, we have
\begin{equation}
\label{eq:tau_Astar}
A^\star p = \int_Q \Aper(p+\nabla w_p) = A_0 p + \int_Q \polarization_p.
\end{equation}
In view of~\eqref{eq:4}, it is natural to introduce the space
\begin{equation*}
\polarizations_0 = \left\{ \tau \in \polarizations, \quad \tau = 0 \ \text{ a.e. on $Q \setminus Q_0$ } \right\}.
\end{equation*}
By analogy with the quantum theory of scattering~\citep{LIPP1950}, problem~\eqref{eq:4} is usually referred to as the Lippmann--Schwinger equation~\citep{KORR1973, ZELL1973, KRON1974, NEMA1982}. We have the following result:
\begin{lem}
\label{lem:equiv}
Under Assumption~\eqref{eq:hyp2}, problem~\eqref{PB:correcteur} is equivalent to problem~\eqref{eq:4}.
\end{lem}

\begin{proof}
We have already shown that, if $w_p$ is a solution to~\eqref{PB:correcteur}, then $\polarization_p = \left(\Aper-A_0\right)\left(p+\nabla w_p\right)$ is a solution to~\eqref{eq:4}. We now prove the converse statement. Let $\polarization_p$ be a solution to~\eqref{eq:4}. Define $w_p$ up to an additive constant by $\nabla w_p = - \green_0 \ast \tau_p$, i.e. $w_p$ is the (unique up to an additive constant) solution to
\begin{equation}
\label{eq:titi}
-\div\left(A_0\nabla w_p+\polarization_p\right)=0\quad\text{in $\RR^d$},
\qquad
w_p\text{ is $\ZZ^d$-periodic}.
\end{equation}
We have from the first line of~\eqref{eq:4} that
\begin{equation*}
\left(\Aper-A_0\right)^{-1}\polarization_p 
=
p - \green_0 \ast \tau_p
=
p + \nabla w_p \quad \text{a.e. on $Q_0$},
\end{equation*}
hence
\begin{equation*}
\polarization_p = \left(\Aper-A_0\right) \left( p + \nabla w_p \right) \quad \text{a.e. on $Q_0$}.
\end{equation*}
In view of the second line of~\eqref{eq:4}, we see that the above relation is actually true on $Q$. Collecting this relation with~\eqref{eq:titi}, we deduce that $w_p$ is a solution to~\eqref{PB:correcteur}.
\end{proof}

\begin{remark}
After discretization, problem~\eqref{eq:4} may be more or less difficult to solve, depending on the choice of $A_0$. Indeed, the condition number of the matrix of the linear system associated to~\eqref{eq:4} depends on $A_0$. The practical performance (for various choices of $A_0$) of the approach described here has been investigated in e.g.~\cite{MICH2001,BRIS2010A}.
\end{remark}

The present approach is motivated by the fact that a problem of type~\eqref{eq:1}, with a {\em constant} matrix $A_0$, is easier to solve than a problem of type~\eqref{PB:correcteur}. In addition, the Lippmann--Schwinger equation~\eqref{eq:4} can be discretized in the form of a {\em matrix-free} problem. See end of Section~\ref{sec:galerkin}.

\section{Numerical discretization of the Lippmann--Schwinger equation}
\label{sec:disc}

In this section, we present a Galerkin discretization of the Lippmann--Schwinger equation~\eqref{eq:4}, and detail the approach down to the identification of the linear system to solve (see Section~\ref{sec:galerkin}). We next discuss practical approximations of the bilinear form (see Section~\ref{sec:approx}) before briefly turning to implementation details in Section~\ref{sec:implem}. To simplify the notation, we omit the dependency of $\tau_p$ with respect to $p$.

\medskip

The weak form of~\eqref{eq:4} is (see~\citep{BRIS2012A})
\begin{equation}
  \label{eq:5}
  \text{Find $\trial\in\polarizations_0$ such that, for all $\test\in\polarizations_0$,} \quad a(\trial, \test)=\int_\unitcellzero \test^Tp,
\end{equation}
where $a$ is the following symmetric bilinear form:
\begin{align}
\label{eq:def_a}
a(\trial, \test)&=\int_\unitcellzero \test^T\left(\Aper-A_0\right)^{-1}\trial+\int_\unitcellzero \test^T\left(\green_0\ast\trial\right)
\\
\nonumber
  &=\int_\unitcellzero \test^T\left(\Aper-A_0\right)^{-1}\trial+\sum_{k\in\ZZ^d} \left(\overline{\hat{\test}(k)}\right)^T \, \hat\green_0(k) \, \hat\trial(k),
\end{align}
where Parseval's identity has been used in the second line and $\left(\overline{\hat{\test}(k)}\right)^T$ is the transpose of the complex conjugate of the vector $\test(k)$. Problem~\eqref{eq:5} is of course well-posed, due to its equivalence to~\eqref{eq:4}, Lemma~\ref{lem:equiv} and the well-posedness of~\eqref{PB:correcteur}. However, as a preparation for the analysis of its discretized counterpart that we describe below, it is also interesting to directly analyze~\eqref{eq:5}. This is the object of Section~\ref{sec:math1} below.

\subsection{Galerkin approximation}
\label{sec:galerkin}

To discretize~\eqref{eq:5}, we follow a standard Galerkin procedure and introduce the problem
\begin{equation}
  \label{eq:6}
  \text{Find $\trial^h\in\approximations_0$ such that, for all $\test^h\in\approximations_0$,} \quad a(\trial^h, \test^h)=\int_\unitcellzero \left(\test^h\right)^Tp,
\end{equation}
where
\begin{equation*}
\approximations_0 = \left\{ \trial^h \in \approximations, \quad \trial^h = 0 \ \text{ a.e. on $Q \setminus Q_0$ } \right\},
\end{equation*}
and where the approximation space $\approximations$ is defined as the space of vector-valued functions $\trial^h \in V$ that are constant over each cell $\cell{\beta}$,
\begin{equation*}
\cell{\beta}
=
h \Big( \beta + \unitcell \Big)
=
\left[\beta_1h,(\beta_1+1)h\right]\times\cdots\times\left[\beta_dh,(\beta_d+1)h\right], \quad \beta\in\ZZ^d,
\end{equation*}
where $h=1/N$ is the size of the cells. The total number of cells in the unit-cell $\unitcell$ is $N^d$. For future use, we introduce
\begin{equation*}
\I = \left\{ \beta \in \ZZ^d, \quad \cell{\beta} \subset \unitcell \right\},
\qquad
\Izero = \left\{ \beta \in \ZZ^d, \quad \cell{\beta} \subset \unitcellzero \right\}.
\end{equation*}
Remark that the discretization of $\unitcell$ into cells $\cell{\beta}$ does not have to be consistent with the subset $\unitcellzero \subset \unitcell$, i.e. we do not require cells $\cell{\beta}$ to be subsets of either $\unitcellzero$ or $\unitcell \setminus \unitcellzero$. As soon as $\cell{\beta} \not\subset \unitcellzero$, i.e. $\beta \notin \Izero$, we have $\trial^h = 0$ on $\cell{\beta}$.

\medskip

In what follows, we provide explicit expressions for the left and right hand sides of~\eqref{eq:6}. For any $\trial^h\in\approximations$ and $\beta\in \I$, let $\trial_\beta^h$ denote the constant value of $\trial^h$ over the cell $\cell{\beta}$. Since $\trial^h$ is $\ZZ^d$-periodic, $\trial_\beta^h$ is $N\ZZ^d$-periodic. The discrete variational problem~\eqref{eq:6} reduces to a linear system with unknowns $\trial_{\beta}^h$, $\beta\in\Izero$, the identification of which requires the evaluation of the bilinear form $a$ over $\approximations_0 \times \approximations_0$, as well as the identification of the linear form in the right-hand side of~\eqref{eq:6} over $\approximations_0$. 

For the first term of $a$, we have
\begin{equation}
\label{eq:7}
\int_\unitcellzero \left(\test^h\right)^T\left(\Aper-A_0\right)^{-1}\trial^h
=
h^d \sum_{\beta \in \Izero} \left(\test_\beta^h\right)^T\left(A_\beta^h-A_0\right)^{-1}\trial_\beta^h
\end{equation}
where $A_\beta^h$ is defined by
\begin{equation}
\label{eq:def_ah}
\left(A_\beta^h-A_0\right)^{-1}=h^{-d}\int_{\cell{\beta}}\left(\Aper-A_0\right)^{-1}.
\end{equation}
For the second term of $a$, straightforward calculations lead to the following expression (see~\cite{BRIS2010A}):
\begin{equation}
\label{eq:8}
\forall \trial^h, \test^h\in\approximations_0, \quad
\int_\unitcellzero \left(\test^h\right)^T\left(\green_0\ast\trial^h\right)
=
\int_\unitcell \left(\test^h\right)^T\left(\green_0\ast\trial^h\right)
=
\frac{1}{N^{2d}} \sumvoxels{k} \left(\overline{\hat\test_k^h}\right)^T \, \green_{0,k}^{h, \text{cons}} \ \hat\trial_k^h.
\end{equation}
In the above formula, $(\hat\test_k^h)_{k \in \ZZ^d}$ is the discrete Fourier transform (DFT) of the $N\ZZ^d$-periodic function $\beta \mapsto \test_\beta^h$, i.e. $\hat\test_k^h$ is the Fourier coefficient of $\test_\beta^h$:
\begin{equation*}
  \hat\test_k^h=\DFT_k\left[ \test^h \right]=\sumvoxels{\beta}\test_\beta^h \, \exp\left(-2i\pi\frac{\beta^Tk}N\right),
\end{equation*}
and likewise for $\hat\trial_k^h$, while $\hat\green_{0, k}^{h,\text{cons}}$ denotes the so-called \emph{consistent discrete Green operator}
\begin{equation}
\label{eq:11}
\hat\green_{0, k}^{h, \text{cons}}
=\sum_{n\in\ZZ^d} \left( F\left[2\pi\left(\frac{k}{N}+n\right)\right] \right)^2 \, \hat\green_0\left(\frac{k}{N}+n\right)\quad\text{with}\quad F(K)=\frac{\sin(K_1/2)}{K_1/2}\cdots\frac{\sin(K_d/2)}{K_d/2}\quad \text{for any $K\in\RR^d$}.
\end{equation}
We observe that $\left(\hat\green_{0, k}^{h, \text{cons}}\right)_{k\in\ZZ^d}$ as defined above is $N\ZZ^d$-periodic. Therefore $\left(\hat\green_{0, k}^{h, \text{cons}} \ \hat\trial_k^h\right)_{k\in\ZZ^d}$ is also $N\ZZ^d$-periodic. Using Plancherel's identity, we recast~\eqref{eq:8} as
\begin{equation}
\label{eq:10}
\int_\unitcellzero\left(\test^h\right)^T\left(\green_0\ast\trial^h\right)
=
h^d\sumvoxels{\beta} \left(\test_\beta^h\right)^T \ \DFT_\beta^{-1}\left[\left(\hat\green_{0, k}^{h, \text{cons}} \, \hat\trial_k^h\right)_{k\in\ZZ^d}\right].
\end{equation}
We now turn to the linear form in the right-hand side of~\eqref{eq:6}, which reads
\begin{equation}
  \label{eq:9}
\forall \test^h \in \approximations_0, \quad
\int_\unitcellzero \left(\test^h\right)^Tp
=
\int_\unitcell \left(\test^h\right)^Tp
=
h^d\sumvoxels{\beta}p^T\test_\beta^h.
\end{equation}
Collecting~\eqref{eq:def_a}, \eqref{eq:7}, \eqref{eq:10} and~\eqref{eq:9}, we obtain the equivalent formulation of~\eqref{eq:6}:
\begin{multline*}
\text{Find $\left(\trial_\beta^h\right)_{\beta\in\Izero}$ such that, for any $\left(\test_\beta^h\right)_{\beta\in\Izero}$,} 
\\
\sum_{\beta \in \Izero} \left( \test_\beta^h \right)^T \left(A_\beta^h-A_0\right)^{-1} \trial_\beta^h
+
\sumvoxels{\beta} \left( \test_\beta^h \right)^T \ \DFT_\beta^{-1}\left[\left(\hat\green_{0, k}^{h, \text{cons}} \, \hat\trial_k^h\right)_{k\in\ZZ^d}\right]
= \sumvoxels{\beta}\left(\test_\beta^h\right)^Tp.
\end{multline*}
Therefore, the linear system to solve is
\begin{equation}
  \label{eq:13}
\forall \beta \in \Izero, \quad
  \left(A_\beta^h-A_0\right)^{-1}\trial_\beta^h+\DFT_\beta^{-1}\left[\left(\hat\green_{0, k}^{h, \text{cons}} \ \hat\trial_k^h\right)_{k\in\ZZ^d}\right]=p.
\end{equation}
Using the circular convolution theorem, we recast the above equation as
\begin{equation}
  \label{eq:12}
\forall \beta \in \Izero, \quad
  \left(A_\beta^h-A_0\right)^{-1}\trial_\beta^h+\sum_{\gamma\in\cellindices}\green_{0, \beta-\gamma}^{h, \text{cons}}\trial_\gamma^h=p,
\end{equation}
where $\left( \green_{0,\beta}^{h, \text{cons}} \right)_{\beta\in\ZZ^d}$ is the ($N\ZZ^d$-periodic) inverse discrete Fourier transform of $\left(\hat\green_{0,k}^{h, \text{cons}}\right)_{k\in\ZZ^d}$ defined by~\eqref{eq:11}. Equation~\eqref{eq:12} shows that $\green_{0,\beta-\gamma}^{h, \text{cons}}$ is the term coupling $\trial_\beta^h$ with $\trial_\gamma^h$.

\medskip

Note that, in~\eqref{eq:13}, the discrete convolution product between $\green_{0}^{h, \text{cons}}$ and $\trial^h$ is computed in the Fourier space, which is much more efficient than in the real space for large discretization grids. Furthermore, problem~\eqref{eq:13} can be implemented as a \emph{matrix-free} linear system, in combination with an iterative linear solver (see~\cite{BRIS2010A}).

\subsection{Asymptotically consistent approximations of the bilinear form}
\label{sec:approx}

We remark that the consistent, discrete Green operator~\eqref{eq:11} that appears in~\eqref{eq:13} is ill-suited to numerical implementation. Indeed, it is defined as a series which actually converges very slowly. This forbids its evaluation at each iteration of the linear solver. Thus, this operator must be stored in memory, which is not practical for large, 3D simulations (see~\cite{BRIS2010A} for examples of 2D applications).

However, it is well-known that the \emph{exact} evaluation of the bilinear form $a$ over $\approximations_0 \times \approximations_0$ is not mandatory to define a converging numerical discretization. Instead, an \emph{asymptotically consistent} approximation $a^h$ of $a$ can be introduced, following~\cite[Definition 2.15]{ERN2004}. We then introduce a further approximate discrete variational problem, which reads (compare with~\eqref{eq:6})
\begin{equation}
  \label{eq:16}
\text{Find $\trial^h\in\approximations_0$ such that, for all $\test^h\in\approximations_0$,} \quad a^h(\trial^h, \test^h)=\int_\unitcellzero \left(\test^h\right)^Tp.
\end{equation}
To preserve the structure (block-diagonal and block-circulant) of the linear system~\eqref{eq:13}, $a^h$ is often defined by
\begin{equation*}
a^h(\trial^h, \test^h)=
h^d \sum_{\beta \in \Izero} \left( \test_\beta^h \right)^T \left(A_\beta^h-A_0\right)^{-1} \trial_\beta^h
+
h^d \sumvoxels{\beta} \left( \test_\beta^h \right)^T \ \DFT_\beta^{-1}\left[\left(\hat\green_{0, k}^h \ \hat\trial_k^h\right)_{k\in\ZZ^d}\right]
\end{equation*}
which only differs from $a(\trial^h, \test^h)$ by the fact that we have replaced $\hat\green_{0, k}^{h, \text{cons}}$ by some $\hat\green_{0, k}^h$ in the second term (compare the above formula with~\eqref{eq:10}). We think of $\hat\green_{0,k}^h$ as a suitable approximation of the consistent, discrete Green operator $\hat\green_{0, k}^{h,\text{cons}}$. Several choices for this approximate discrete Green operator have been introduced in the literature, that we now review. A comparison of these choices is then provided. It should be noted that the terminology adopted here (namely truncated, filtered, finite-difference-based, finite-element-based Green operator) is ours.

\subsubsection{The truncated Green operator}

The first asymptotically consistent approximation of~\eqref{eq:5} was proposed by Moulinec and Suquet~\citep{suquet1, suquet2} who simply discarded the high-frequency terms in~\eqref{eq:2}. This results in the following definition of the truncated Green operator $\hat\green_{0,k}^{h,\text{trunc}}$: for even values of $N=2M$, 
\begin{equation}
\label{eq:14}
\forall n\in\ZZ^d, \quad \forall k \in \freqindiceseven, \quad \hat\green_{0,k+nN}^{h,\text{trunc}} := \hat\green_0(k),
\end{equation}
whereas, for odd values of $N=2M+1$,
\begin{equation}
\label{eq:15}
\forall n\in\ZZ^d, \quad \forall k \in \freqindicesodd, \quad \hat\green_{0,k+nN}^{h,\text{trunc}} := \hat\green_0(k).
\end{equation}
It has been shown in~\citep{BRIS2012A} that this discrete operator has the asymptotic consistency property as defined in~\citep[Definition 2.15]{ERN2004}.

\medskip

It should be noted that, when $N$ is even, Equation~\eqref{eq:14} does not define the discrete Fourier transform of \emph{real} numbers. Indeed, it can be verified that the symmetry property $\dps \hat\green_{0, k}^{h,\text{trunc}}=\overline{\hat\green_{0,N-k}^{h,\text{trunc}}}$ does \emph{not} hold as soon as one of the $k_i$'s is equal to $M$. For such values of the discrete frequency $k$, Moulinec and Suquet~\cite{suquet2} replace~\eqref{eq:14} with $\dps \hat\green_{0, k}^{h,\text{trunc}} := \left(A_0\right)^{-1}$.

\subsubsection{The filtered discrete Green operator}
\label{sec:filtered}

Instead of cutting-off the high-frequency terms in~\eqref{eq:2} as above, it can be advantageous to \emph{smoothly} filter them out. Such a filtering approach has been proposed in~\cite{BRIS2012A}, and results in the following approximation of $\hat\green_{0,k}^{h,\text{cons}}$, which should be compared with~\eqref{eq:11}:
\begin{equation*}
\hat\green_{0, k}^{h, \text{filt}} := \sum_{n\in\{-1,0\}^d} \left( G\left[2\pi \left( \frac{k}{N}+n \right)\right] \right)^2 \, \hat\green_0 \left( \frac{k}{N}+n\right)\quad\text{with}\quad G(K)=\cos\frac{K_1}4\cdots\cos\frac{K_d}4 \quad \text{for any $K\in\RR^d$}.
\end{equation*}

\subsubsection{Discrete Green operator based on a Finite Difference discretization of~\eqref{eq:1}}

A discrete Green operator based on a finite difference discretization of~\eqref{eq:1} has been introduced in~\cite{WILL2014}. Its expression is given by
\begin{equation*}
  \hat\green_{0, k}^{h, \text{FD}}=\frac{V_k \left(\overline{V_k} \right)^T}{\left(\overline{V_k} \right)^T A_0 V_k}\quad\text{with}\quad V_k=\left[\exp\left(\frac{2i\pi k_1}N\right)-1,\ldots,\exp\left(\frac{2i\pi k_d}N\right)-1\right]^T,
\end{equation*}
where $\left(\overline{V_k} \right)^T$ denotes the conjugate transpose of the vector $V_k$. 

\subsubsection{Discrete Green operator based on a Finite Element discretization of~\eqref{eq:1}}

We finally mention the approach proposed by Yvonnet in~\cite{YVON2012}. It amounts to solving~\eqref{eq:1} \emph{in the real space}, using finite elements. Indeed, convolution products with small kernels can arguably be computed more efficiently in the real space than by means of fast Fourier transforms. This results in a compactly supported approximation of the Green operator. Of course, no explicit formula for that approximation is available.

\subsubsection{Comparison}

All the discrete Green operators presented above have strengths and weaknesses, and none clearly stands out. In Table~\ref{tab:1}, we compare these operators according to the following criteria:
\begin{enumerate}
\item {\bf Ease of evaluation}: since discrete Green operators are used in conjunction with \emph{iterative} linear solvers, it is required that they be efficiently computed. It should be noted that for large-scale, 3D simulations, storage of this operator (as required e.g. for the FEM-based discrete Green operator introduced in~\cite{YVON2012}) may become challenging, and that additional assumptions may be required to make the storage affordable in practice.
\item {\bf Independence with respect to $A_0$}: let $\trial^h$ be the solution to the asymptotically consistent discrete variational problem \eqref{eq:16}. The numerical scheme is said to be independent with respect to $A_0$ if the estimate $\left(\Aper-A_0\right)^{-1}\trial^h$ of $p + \nabla w_p$ (where $w_p$ is the exact solution to the corrector problem~\eqref{PB:correcteur}) does not depend on $A_0$ (of course, $\trial^h$ \emph{does} depend on $A_0$). Such feature is highly desirable, as the accuracy of the numerical scheme is in this case not affected by the choice of the reference material $A_0$, which is a free, user-chosen parameter. Note that, in contrast, the convergence of the iterative linear solver strongly depends on $A_0$, as first noted in~\citep{MICH2001}.
\item {\bf Smoothness of numerical solution}: discretization of~\eqref{eq:5} amounts to neglecting high-frequencies in the solution. Depending on how precisely these high-frequencies are neglected, spurious oscillations (``checkerboard pattern'') may appear in the solution, as noted in~\cite{WILL2008, BRIS2012A, WILL2014}.
\end{enumerate}

From Table~\ref{tab:1}, we see that the truncated discrete Green operator may lead to spurious oscillations, while the numerical scheme resulting from the filtered discrete Green operator depends on the reference material. It is then up to the user to find the best reference material $A_0$, which minimizes the numerical error on the estimate of $\nabla w_p$, the grid-size being fixed. The discrete Green operator based on finite differences does not suffer from these deficiencies. Unfortunately, as argued by the authors themselves, this approximation, introduced in~\cite{WILL2014} in the setting of a scalar problem, does not extend well to linear elasticity (see~\cite{WILL2008}).

\begin{table}[htbp]
  \centering
  \begin{tabular}{l|c|c|c|c}
    & Truncated & Filtered & Finite Differences & Finite Elements\\
    \hline
    Efficiently computed / Easy to store & Yes & Yes & Yes & No\\
    Solution is independent of $A_0$ & Yes & No & Yes & No\\
    Solution is smooth & No & Yes & Yes & Yes\\
    References & \cite{suquet1, suquet2} & \cite{BRIS2012A} & \cite{WILL2008, WILL2014} & \cite{YVON2012}
  \end{tabular}
  \caption{Comparison of the four discrete Green operators presented above. The comparison is performed according to their state of the art description, and does not take into account possible future improvements.}
  \label{tab:1}
\end{table}

\subsection{Implementation of the method}
\label{sec:implem}

Within the framework introduced above, the Galerkin discretization of the Lippmann--Schwinger equation can readily be implemented as a \emph{matrix-free} method. In other words, the linear operator
\begin{equation*}
\trial^h \mapsto \left[ \left(A_\beta^h-A_0\right)^{-1}\trial_\beta^h+\sum_{\gamma\in\cellindices}\green_{0, \beta-\gamma}^{h, \text{cons}}\trial_\gamma^h \right]_{\beta \in \I}
\end{equation*}
is defined as a \emph{function}, and not as a matrix. The first term is purely local to each voxel. It is readily parallelized. The second term is evaluated in Fourier space, by means of the parallel version of the \texttt{FFTW3} library~\cite{FRIG2005}. The parallel iterative linear solvers of the \texttt{PETSc} library can then be used~\cite{petsc-web-page, petsc-user-ref, petsc-efficient}.

Our implementation follows a modular, object-oriented approach, which allows to easily switch between various discrete Green operators as well as various physical settings, including the scalar setting described in this article, as well as the linear elasticity setting. Extension to Darcy flows~\cite{NGUY2013} should also be possible with limited effort.

\section{Mathematical analysis}
\label{sec:math}

The aim of this section is three-fold. 

First, we show that~\eqref{eq:5} is well-posed (see Section~\ref{sec:math1}). As pointed out above, this is of course true due to the equivalence of~\eqref{eq:5} with~\eqref{eq:4}, Lemma~\ref{lem:equiv} and the well-posedness of~\eqref{PB:correcteur}. However, the direct analysis of~\eqref{eq:5}, without using the fact that this problem is equivalent to the corrector problem~\eqref{PB:correcteur}, is actually useful as a preparation for the analysis of its discretized counterpart.

Second, in Section~\ref{sec:math2}, we turn to the discrete problem~\eqref{eq:6}, and show its well-posedness. 

Third, we turn to estimating the error between the solutions of~\eqref{eq:5} and~\eqref{eq:6} in Section~\ref{sec:math3}.

\medskip

Our analysis, which closely follows that of~\cite{BRIS2012A}, is performed under the following assumption. Recall that we have assumed $\Aper(x)$ and $A_0$ to be symmetric matrices. We further assume that
\begin{equation}
\label{eq:hyp3}
\text{$\Aper(x)$ and $A_0$ commute a.e. on $\unitcellzero$.}
\end{equation}
To date, it is unclear to us whether this assumption is technical or essential for the mathematical analysis.

The matrices $\Aper(x)$ and $A_0$ are thus simultaneously diagonalizable, and there exists $c>0$, $\lambda^{\rm per}_i(x)$, $\lambda^0_i(x)$ and $\psi_i(x)$, $1 \leq i \leq d$, such that, almost everywhere on $\unitcellzero$,
\begin{equation*}
\Aper(x) \, \psi_i(x) = \lambda^{\rm per}_i(x) \, \psi_i(x), 
\quad
A_0 \, \psi_i(x) = \lambda^0_i(x) \, \psi_i(x)
\quad
\lambda^{\rm per}_i(x) \geq c, 
\quad
\lambda^0_i(x) \geq c. 
\end{equation*}
We choose the eigenvectors such that, a.e. on $\unitcellzero$, $\left\{ \psi_i(x) \right\}_{1 \leq i \leq d}$ forms an orthonormal basis of $\RR^d$. In view of~\eqref{eq:hyp2}, \eqref{eq:hyp1} and of the fact that $\Aper \in L^\infty$, we also have that
\begin{equation}
\label{eq:bing}
C \geq \Big| \lambda^{\rm per}_i(x) - \lambda^0_i(x) \Big| \geq c_0 > 0 \quad \text{a.e. on $\unitcellzero$}.
\end{equation}
 
\begin{remark}
Note that, if the reference material is isotropic, i.e. the matrix $A_0$ is given by $A_0 = a_0 \, \text{Id}$ for some $a_0 > 0$, then Assumption~\eqref{eq:hyp3} holds. In the case discussed in Remark~\ref{rem:random} of random materials, choosing $A_0 = a_0 \, \text{Id}$ would be a natural choice if the heterogeneous materials is statistically isotropic.
\end{remark}

\subsection{Problem~\eqref{eq:5} is well-posed}
\label{sec:math1}

We follow here the ideas of~\cite{BRIS2012A} and show that~\eqref{eq:5} satisfies the assumptions of the BNB theorem as formulated in~\cite[Theorem 2.6]{ERN2004}. For any $\tau \in V_0$, we define $\tau_+ \in V_0$ and $\tau_- \in V_0$ by
\begin{gather*}
\tau_+(x) = \sum_{i=1}^d 1_{\lambda^{\rm per}_i(x) > \lambda^0_i(x)} \ \Big( (\tau(x))^T \psi_i(x) \Big) \ \psi_i(x), 
\quad 
\tau_-(x) = \sum_{i=1}^d 1_{\lambda^{\rm per}_i(x) < \lambda^0_i(x)} \ \Big( (\tau(x))^T \psi_i(x) \Big) \ \psi_i(x)
\quad
\text{a.e. on $\unitcellzero$,}
\\
\tau_+(x) = \tau_-(x) = 0 \quad \text{a.e. on $\unitcell \setminus \unitcellzero$.}
\end{gather*}
We have
\begin{equation*}
\tau = \tau_+ + \tau_-, \quad \tau_+ \in V_0, \quad \tau_- \in V_0.
\end{equation*}
We introduce
\begin{equation}
\label{eq:def_test}
\test = \tau_+ - \tau_- \in V_0
\end{equation}
and have the following result:
\begin{lem}
\label{lem:cond1}
We assume that~\eqref{eq:hyp1}, \eqref{eq:hyp2} and~\eqref{eq:hyp3} hold. For any $\tau \in V_0$, let $\test \in V_0$ be defined by~\eqref{eq:def_test}. There exists $c > 0$ independent of $\tau$ and $\test$ such that
\begin{equation*}
a(\tau,\test) \geq c \| \tau \|_{L^2(\unitcellzero)} \| \test \|_{L^2(\unitcellzero)}.
\end{equation*}
\end{lem}

\begin{proof}
We write
\begin{equation}
\label{eq:resu0}
a(\tau,\test) = a(\tau_++\tau_-,\tau_+-\tau_-) = a(\tau_+,\tau_+)-a(\tau_-,\tau_-),
\end{equation}
where we have used that $a$ is symmetric. We subsequently study the two terms of the above right-hand side.

\medskip

\noindent
\textbf{Step 1:} We have
\begin{equation}
\label{eq:tata1}
a(\tau_+,\tau_+)
=
\int_\unitcellzero \tau_+^T \left(\Aper-A_0\right)^{-1} \tau_+ + \int_\unitcellzero \tau_+^T \left(\green_0\ast\tau_+\right).
\end{equation}
Let $u_+$ be the unique (up to the addition of a constant) solution to
\begin{equation*}
-\div\left(A_0\nabla u_++\polarization_+\right)=0\quad\text{in $\RR^d$},
\qquad
u_+\text{ is $\ZZ^d$-periodic},
\end{equation*}
so that $\nabla u_+ = - \green_0 \ast \polarization_+$. Then the second term of $a(\tau_+,\tau_+)$ satisfies
\begin{equation}
\label{eq:tata2}
\int_\unitcellzero \tau_+^T \left(\green_0\ast\tau_+\right)
=
- \int_\unitcell \tau_+^T \nabla u_+
=
\int_\unitcell u_+ \div \tau_+
=
- \int_\unitcell u_+ \div \left( A_0 \nabla u_+ \right)
=
\int_\unitcell (\nabla u_+)^T A_0 \nabla u_+.
\end{equation}
For the first term of $a(\tau_+,\tau_+)$, we use our specific construction of $\tau_+$ and write, using~\eqref{eq:bing}, that
\begin{eqnarray}
\nonumber
\tau_+(x)^T \left(\Aper(x)-A_0\right)^{-1} \tau_+(x)
&=&
\sum_{i=1}^d \sum_{j=1}^d 1_{\lambda^{\rm per}_i(x) > \lambda^0_i(x)} \ 1_{\lambda^{\rm per}_j(x) > \lambda^0_j(x)} \ \Big( (\tau(x))^T \psi_i(x) \Big) \ \Big( (\tau(x))^T \psi_j(x) \Big) \ \frac{\Big( \psi_j(x) \Big)^T \psi_i(x)}{\lambda^{\rm per}_i(x) - \lambda^0_i(x)}
\\
\nonumber
&=&
\sum_{i=1}^d 1_{\lambda^{\rm per}_i(x) > \lambda^0_i(x)} \ \frac{\Big( (\tau(x))^T \psi_i(x) \Big)^2}{\lambda^{\rm per}_i(x) - \lambda^0_i(x)}
\\
\nonumber
&\geq&
C^{-1} \sum_{i=1}^d 1_{\lambda^{\rm per}_i(x) > \lambda^0_i(x)} \ \Big( (\tau(x))^T \psi_i(x) \Big)^2
\\
\label{eq:tata3}
&=&
C^{-1} \Big| \tau_+(x) \Big|^2.
\end{eqnarray}
Collecting~\eqref{eq:tata1}, \eqref{eq:tata2} and~\eqref{eq:tata3}, we deduce that
\begin{equation}
\label{eq:resu1}
a(\tau_+,\tau_+)
\geq
C^{-1} \left\| \tau_+ \right\|^2_{L^2(Q_0)} + \int_\unitcell (\nabla u_+)^T A_0 \nabla u_+
\geq
C^{-1} \left\| \tau_+ \right\|^2_{L^2(Q_0)}.
\end{equation}

\medskip

\noindent
\textbf{Step 2:} Let $u_-$ be the unique (up to the addition of a constant) solution to
\begin{equation*}
-\div\left(A_0\nabla u_-+\polarization_-\right)=0\quad\text{in $\RR^d$},
\qquad
u_-\text{ is $\ZZ^d$-periodic},
\end{equation*}
so that $\nabla u_- = - \green_0 \ast \polarization_-$. We introduce the divergence free vector-valued field
\begin{equation*}
\Sigma = A_0 \nabla u_- + \tau_-,
\end{equation*}
and the vector-valued field
\begin{equation}
\label{eq:def_eta}
\eta := A_0^{-1} \tau_- = A_0^{-1} \Sigma - \nabla u_-.
\end{equation}
By definition,
\begin{equation}
\label{eq0}
a(\tau_-,\tau_-)
=
\int_\unitcellzero \tau_-^T \left(\Aper-A_0\right)^{-1} \tau_- + \int_\unitcellzero \tau_-^T \left(\green_0\ast\tau_-\right).
\end{equation}
We successively consider the two terms of the right-hand side. For the second term, we see, using the definitions of $\green_0$ and $\eta$, that
\begin{equation*}
\int_\unitcellzero \tau_-^T \left(\green_0\ast\tau_-\right)
=
-\int_\unitcell \tau_-^T \nabla u_-
=
-\int_\unitcell \eta^T A_0 \nabla u_-
=
\int_\unitcell \eta^T A_0 \eta - \int_\unitcell \eta^T \Sigma
=
\int_\unitcell \eta^T A_0 \eta - \int_\unitcell \left( A_0^{-1} \Sigma - \nabla u_- \right)^T \Sigma.
\end{equation*}
Since $\Sigma$ is divergence free and $A_0$ is symmetric, we obtain
\begin{equation}
\label{eq1}
\int_\unitcellzero \tau_-^T \left(\green_0\ast\tau_-\right)
=
\int_\unitcell \eta^T A_0 \eta - \int_\unitcell \Sigma^T A_0^{-1} \Sigma.
\end{equation}
We now consider the first term in the right-hand side of~\eqref{eq0}. We have, for any $x \in Q_0$, that
\begin{multline*}
\text{Id}
=
\left[\Aper^{-1} (\Aper-A_0)\right]^{-1} \left[\Aper^{-1} (\Aper-A_0)\right]
=
\left[\Aper^{-1} (\Aper-A_0)\right]^{-1} - \left[\Aper^{-1} (\Aper-A_0)\right]^{-1} \Aper^{-1} A_0
\\
=
\left[\Aper^{-1} (\Aper-A_0)\right]^{-1} - \left[ \Aper-A_0 \right]^{-1} A_0
\end{multline*}
where $\text{Id}$ is the identity matrix. We thus deduce that
\begin{equation*}
\left[ \Aper-A_0 \right]^{-1} A_0
=
\left[\Aper^{-1} (\Aper-A_0)\right]^{-1} - \text{Id}
=
\left[ \left( A_0^{-1} - \Aper^{-1} \right) A_0 \right]^{-1} - \text{Id}
=
A_0^{-1} \left[ \left( A_0^{-1} - \Aper^{-1} \right) \right]^{-1} - \text{Id}.
\end{equation*}
We hence get that, for any $x \in Q_0$,
\begin{multline}
\label{eq2}
\tau_-(x)^T \left(\Aper(x)-A_0\right)^{-1} \tau_-(x)
=
\eta(x)^T A_0 \left(\Aper(x)-A_0\right)^{-1} A_0 \eta(x)
=
\eta(x)^T A_0 \left[A_0^{-1} \left( A_0^{-1} - \Aper(x)^{-1} \right)^{-1} -\text{Id}\right] \eta(x)
\\
=
\eta(x)^T \left[ \left( A_0^{-1} - \Aper(x)^{-1} \right)^{-1} - A_0 \right] \eta(x).
\end{multline}
Collecting~\eqref{eq0}, \eqref{eq1} and~\eqref{eq2}, we obtain
\begin{eqnarray}
a(\tau_-,\tau_-)
&=&
\int_\unitcell \eta^T A_0 \eta - \int_\unitcell \Sigma^T A_0^{-1} \Sigma
+
\int_\unitcellzero \eta^T \left[ \left( A_0^{-1} - \Aper^{-1} \right)^{-1} - A_0 \right] \eta
\nonumber
\\
&=&
- \int_\unitcell \Sigma^T A_0^{-1} \Sigma
+
\int_{\unitcellzero} \eta^T \left( A_0^{-1} - \Aper^{-1} \right)^{-1} \eta.
\label{eq:tata5}
\end{eqnarray}
We now use our specific choice for $\tau_-$, and write
\begin{eqnarray}
\nonumber
&& \eta(x)^T \left( A_0^{-1} - \Aper(x)^{-1} \right)^{-1} \eta(x)
\\
\nonumber
&=&
\tau_-(x)^T A_0^{-1} \left(A_0^{-1} - \Aper(x)^{-1} \right)^{-1} A_0^{-1} \tau_-(x)
\\
\nonumber
&=&
\sum_{i=1}^d \sum_{j=1}^d 1_{\lambda^{\rm per}_i(x) < \lambda^0_i(x)} \ 1_{\lambda^{\rm per}_j(x) < \lambda^0_j(x)} \ \Big( (\tau(x))^T \psi_i(x) \Big) \ \Big( (\tau(x))^T \psi_j(x) \Big) \ \frac{\Big( \psi_j(x) \Big)^T \psi_i(x)}{\Big( \lambda^0_i(x) \Big)^2 \ \Big[ \Big( \lambda^0_i(x) \Big)^{-1} - \Big( \lambda^{\rm per}_i(x) \Big)^{-1} \Big]}
\\
\nonumber
&=&
\sum_{i=1}^d 1_{\lambda^{\rm per}_i(x) < \lambda^0_i(x)} \ \frac{\Big( (\tau(x))^T \psi_i(x) \Big)^2}{\Big( \lambda^0_i(x) \Big)^2 \ \Big[ \Big( \lambda^0_i(x) \Big)^{-1} - \Big( \lambda^{\rm per}_i(x) \Big)^{-1} \Big]}
\\
\nonumber
&\leq&
- C \sum_{i=1}^d 1_{\lambda^{\rm per}_i(x) < \lambda^0_i(x)} \ \Big( (\tau(x))^T \psi_i(x) \Big)^2
\\
\label{eq:tata4}
&=&
- C \Big| \tau_-(x) \Big|^2.
\end{eqnarray}
Collecting~\eqref{eq:tata5} and~\eqref{eq:tata4}, we deduce that
\begin{equation}
\label{eq:resu2}
- a(\tau_-,\tau_-)
\geq
C \left\| \tau_- \right\|^2_{L^2(\unitcellzero)} + \int_\unitcell \Sigma^T A_0^{-1} \Sigma
\geq
C \left\| \tau_- \right\|^2_{L^2(\unitcellzero)}
\end{equation}
for some $C>0$.

\medskip

\noindent
\textbf{Conclusion:} Collecting~\eqref{eq:resu0}, \eqref{eq:resu1} and~\eqref{eq:resu2}, we obtain that
\begin{equation*}
a(\tau,\test)
\geq
C \left( \left\| \tau_+ \right\|^2_{L^2(\unitcellzero)} + \left\| \tau_- \right\|^2_{L^2(\unitcellzero)} \right)
=
C \left\| \tau \right\|_{L^2(\unitcellzero)} \left\| \test \right\|_{L^2(\unitcellzero)}
\end{equation*}
for some $C>0$. This concludes the proof of Lemma~\ref{lem:cond1}.
\end{proof}

We are now in position to prove the main result of this section:
\begin{theorem}
\label{theo:cond1}
We assume that~\eqref{eq:hyp1}, \eqref{eq:hyp2} and~\eqref{eq:hyp3} hold. Then problem~\eqref{eq:5} is well-posed.
\end{theorem}

\begin{proof}
We use the BNB theorem as stated in~\cite[Theorem 2.6]{ERN2004}. First, using~\eqref{eq:def_a} and~\eqref{eq:hyp2}, we have
\begin{eqnarray*}
\left| a(\trial, \test) \right|
&\leq&
\left\| \left(\Aper-A_0\right)^{-1} \right\|_{L^\infty(\unitcellzero)} \ \| \test \|_{L^2(\unitcellzero)} \ \| \trial \|_{L^2(\unitcellzero)} + \| \test \|_{L^2(\unitcellzero)} \ \| \green_0\ast\trial \|_{L^2(\unitcellzero)}
\\
&\leq&
c_0^{-1} \ \| \test \|_{L^2(\unitcellzero)} \ \| \trial \|_{L^2(\unitcellzero)} + c \| \test \|_{L^2(\unitcellzero)} \ \| \trial \|_{L^2(\unitcellzero)}
\end{eqnarray*}
which shows that the bilinear form $a$ is continuous on $V_0 \times V_0$.

Second, we infer from Lemma~\ref{lem:cond1} that the inf-sup condition is satisfied:
\begin{equation*}
\inf_{\tau \in V_0} \ \sup_{\test \in V_0} \ \frac{a(\tau,\test)}{\| \tau \|_{L^2(\unitcellzero)} \ \| \test \|_{L^2(\unitcellzero)}} \geq c > 0.
\end{equation*}
The first condition in~\cite[Theorem 2.6]{ERN2004} is therefore satisfied.

Let now $\test \in V_0$ be such that, for any $\tau \in V_0$, we have $a(\tau,\test) = 0$. Since $a$ is symmetric, this implies that $a(\test,\tau) = 0$ for any $\tau$. We then deduce from Lemma~\ref{lem:cond1} that $\test=0$. The second condition in~\cite[Theorem 2.6]{ERN2004} is therefore satisfied.

We are now in position to apply~\cite[Theorem 2.6]{ERN2004}, which allows us to conclude that~\eqref{eq:5} is well-posed.
\end{proof}

\subsection{Problem~\eqref{eq:6} is well-posed}
\label{sec:math2}

We mimick at the discrete stage the proof of Section~\ref{sec:math1}. As pointed out in~\cite{BRIS2012A}, we see, using~\eqref{eq:7}, that
\begin{equation*}
a(\tau^h,\test^h) = \int_\unitcellzero (\test^h)^T\left(A^h-A_0\right)^{-1}\trial^h+\int_\unitcellzero (\test^h)^T\left(\green_0\ast\trial^h\right) 
\end{equation*}
where $A^h$ is a piecewise constant matrix defined by 
\begin{equation*}
\forall \beta \in \Izero, \quad A^h(x) = A_\beta^h \quad \text{in $\cell{\beta}$},
\end{equation*}
where $A_\beta^h$ is defined by~\eqref{eq:def_ah}. Note that the value of $A^h$ on cells $\cell{\beta}$ with $\beta \not\in \Izero$ actually does not enter the expression of $a(\tau^h,\test^h)$. We therefore do not define $A^h$ there. 

\begin{lem}
\label{lem:cond2}
We assume that~\eqref{eq:hyp1}, \eqref{eq:hyp2} and~\eqref{eq:hyp3} hold. For any $\tau \in \approximations_0$, there exists $\test^h \in \approximations_0$ such that
\begin{equation*}
a(\tau^h,\test^h) \geq c \| \tau^h \|_{L^2(\unitcellzero)} \| \test^h \|_{L^2(\unitcellzero)}
\end{equation*}
for some $c>0$ independent of $h$, $\tau^h$ and $\test^h$.
\end{lem}

\begin{proof}
We deduce from~\eqref{eq:hyp3} and~\eqref{eq:def_ah} that $A^h(x)$ commutes with $A_0$. The two matrices are thus simultaneously diagonalizable, and there exists $c>0$, $\mu^h_i(x)$, $\mu^0_i(x)$ and $\varphi_i(x)$, $1 \leq i \leq d$, such that, almost everywhere on $\dps \cup_{\beta \in \Izero} \cell{\beta}$,
\begin{equation*}
A^h(x) \, \varphi_i(x) = \mu^h_i(x) \, \varphi_i(x), 
\quad
A_0 \, \varphi_i(x) = \mu^0_i(x) \, \varphi_i(x)
\quad
\mu^h_i(x) \geq c, 
\quad
\mu^0_i(x) \geq c.
\end{equation*}
Again, we choose the eigenvectors such that, a.e. on $\dps \cup_{\beta \in \Izero} \cell{\beta}$, $\left\{ \varphi_i(x) \right\}_{1 \leq i \leq d}$ forms an orthonormal basis of $\RR^d$. 

Let us now show that, similarly to~\eqref{eq:hyp2}, $A^h$ satisfies, for any $h$,
\begin{equation}
\label{eq:hyp2h}
\forall \xi \in \RR^d, \quad c_0 |\xi|^2 \leq \left| \xi^T \left( A^h(x) - A_0 \right) \xi \right| \quad \text{a.e. on $\dps \cup_{\beta \in \Izero} \cell{\beta}$.}
\end{equation}
For any $\xi \in \RR^d$ and $\beta \in \Izero$, we write, using~\eqref{eq:def_ah} and~\eqref{eq:hyp2}, that
\begin{equation*}
\left| \xi^T \left( A^h_\beta - A_0 \right)^{-1} \xi \right|
=
h^{-d} \left| \int_{\cell{\beta}} \xi^T \left(\Aper-A_0\right)^{-1} \xi \right|
\leq
c_0^{-1} | \xi |^2,
\end{equation*}
which implies~\eqref{eq:hyp2h}. 

We are now left with following the same steps as in Section~\ref{sec:math1}, where $A^h$ plays the role of $\Aper$. For any $\tau^h \in \approximations_0$, we define $\tau^h_+ \in \approximations_0$ and $\tau^h_- \in \approximations_0$ by
\begin{gather*}
\tau^h_+(x) = \sum_{i=1}^d 1_{\mu^h_i(x) > \mu^0_i(x)} \ \Big( (\tau^h(x))^T \varphi_i(x) \Big) \ \varphi_i(x), 
\quad 
\tau^h_-(x) = \sum_{i=1}^d 1_{\mu^h_i(x) < \mu^0_i(x)} \ \Big( (\tau^h(x))^T \varphi_i(x) \Big) \ \varphi_i(x)
\quad
\text{a.e. on $\dps \cup_{\beta \in \Izero} \cell{\beta}$,}
\\
\tau^h_+(x) = \tau^h_-(x) = 0 \quad \text{a.e. on $\dps \unitcell \setminus \cup_{\beta \in \Izero} \cell{\beta}$.}
\end{gather*}
These two functions belong to $V_0$ and are piecewise constant, and therefore belong to $\approximations_0$. We have $\tau^h = \tau^h_+ + \tau^h_-$ and we introduce
\begin{equation*}
\test^h = \tau^h_+ - \tau^h_- \in \approximations_0,
\end{equation*}
which satisfies, as in Lemma~\ref{lem:cond1}, that
\begin{equation*}
a(\tau^h,\test^h) \geq c \| \tau^h \|_{L^2(\unitcellzero)} \| \test^h \|_{L^2(\unitcellzero)}
\end{equation*}
for some $c>0$ independent of $h$, $\tau^h$ and $\test^h$.
\end{proof}

We are now in position to state the main result of this section, the proof of which follows the same lines as the proof of Theorem~\ref{theo:cond1}:
\begin{theorem}
We assume that~\eqref{eq:hyp1}, \eqref{eq:hyp2} and~\eqref{eq:hyp3} hold. Then problem~\eqref{eq:6} is well-posed.
\end{theorem}

\subsection{Error analysis}
\label{sec:math3}

In this section, we estimate the difference between the solution $\tau_p$ to the exact problem~\eqref{eq:5} and its discrete approximation $\tau^h_p$ solution to~\eqref{eq:6}. Following~\cite[Lemma 2.28]{ERN2004} and the fact that the constant $c$ in the inf-sup inequality of the discrete problem is independent of $h$ (see Lemma~\ref{lem:cond2}), we have that
\begin{equation}
\label{eq:cea}
\| \tau_p - \tau^h_p \|_{L^2(\unitcellzero)} \leq C \inf_{\test^h \in \approximations_0} \| \tau_p - \test^h \|_{L^2(\unitcellzero)}
\end{equation}
for some $C$ independent of $h$. We are therefore left with quantifying the best approximation error, and thus to study the regularity of $\tau_p$. To that aim, we first study the regularity of $w_p$ solution to~\eqref{PB:correcteur}, and recall the relation~\eqref{eq:3} between $\nabla w_p$ and $\tau_p$.

If $\Aper$ is H\"older-continuous, then we know that there exists $\alpha > 0$ such that, for any $p \in \RR^d$, we have $w_p \in C^{1,\alpha}(Q)$ (see e.g.~\cite[Theorem 8.22 and Corollary 8.36]{gilbarg-trudinger}). However, the assumption that $\Aper$ is continuous excludes many interesting cases in practice, including composite materials where $\Aper$ is piecewise constant in $Q$. 

We therefore revert to a different setting, namely that of~\cite{li-vogelius}, and assume the following:
\begin{hyp}
\label{eq:hyp4}
There exists a bounded domain $D \subset \RR^d$ with $C^{1,\alpha}$ boundary, $0<\alpha<1$, such that $\overline{Q} \subset\subset D$, and a finite number $M$ of subdomains $D_m$ of $D$, $1 \leq m \leq M$, that are disjoint, such that $\dps \overline{D} = \cup_{1 \leq m \leq M} \overline{D_m}$, and such that all but one of the domains $D_m$ are convex and with $C^2$ boundary. 

For any $1 \leq m \leq M$, let $A^{(m)} \in C^\mu\left(\overline{D_m}\right)$ for some $0<\mu<1$ be a symmetric, matrix-valued function such that $A^{(m)}(x) \geq a_-$ a.e. on $D_m$ (in the sense of symmetric matrices) for some $a_->0$. We assume that the matrix $\Aper$ in~\eqref{PB:correcteur} satisfies the following:
\begin{equation*}
\forall 1 \leq m \leq M, \quad \forall x \in D_m \cap \unitcell, \quad \Aper(x) = A^{(m)}(x).
\end{equation*}
\end{hyp}
This assumption means that, in the unit cell $\unitcell$ complemented by periodicity boundary conditions, there are $M-1$ convex, disjoint inclusions with regular (i.e. $C^2$) boundaries, such that $\Aper$ is H\"older-continuous on each of these inclusions. Outside of these inclusions, $\Aper$ is also H\"older-continuous. By construction, $\Aper$ satisfies~\eqref{eq:hyp1}, is symmetric and in $L^\infty$. We then have the following result:
\begin{theorem}
\label{theo:regu}
We assume that Assumption~\ref{eq:hyp4} holds. Then there exists $\gamma>0$ such that the solution $w_p$ to~\eqref{PB:correcteur} satisfies $w_p \in C^{1,\gamma}\left( \overline{D_m}\right)$ for any $1 \leq m \leq M$ and $p \in \RR^d$.

Furthermore, there exists $\delta>0$ such that the solution $\tau_p$ to~\eqref{eq:5} satisfies $\tau_p \in C^{\delta}\left( \overline{D_m}\right)$ for any $1 \leq m \leq M$ and $p \in \RR^d$.
\end{theorem}

\begin{proof}
The first statement is a direct consequence of~\cite[Theorem 1.1]{li-vogelius}, using the note following~\cite[Theorem 1.2]{li-vogelius}. 

The product of two H\"older-continuous functions is also H\"older-continuous. We thus infer the second statement from the first statement and~\eqref{eq:3}.
\end{proof}

\begin{remark}
\label{rem:regu}
In the two-dimensional case, if $\Aper$ is {\em constant} in the inclusions and outside of the inclusions, then a better regularity on $w_p$ can be shown, see e.g.~\cite[Section 8]{li-vogelius} for some cases in that vein.
\end{remark}

We now turn to the main result of this section:
\begin{theorem}
\label{theo:error}
We assume that~\eqref{eq:hyp1}, \eqref{eq:hyp2}, \eqref{eq:hyp3} and Assumption~\ref{eq:hyp4} hold. We also assume that the boundary of $\unitcell_0$ is regular. Let $\tau_p$ and $\tau^h_p$ be the solutions to the exact problem~\eqref{eq:5} and to the discrete problem~\eqref{eq:6}, respectively. Then there exist $C>0$ and $\delta > 0$ such that, for any $h$,
\begin{equation}
\label{eq:error}
\| \tau_p - \tau^h_p \|_{L^2(\unitcellzero)} \leq C h^{\min(1/2,\,\delta)}.
\end{equation}
In turn, the error on the homogenized coefficients is bounded by
\begin{equation*}
\left| A^\star - A^\star_h \right| \leq C h^{\min(1/2,\,\delta)}
\end{equation*}
where the matrix $A^\star_h$ is the numerical approximation of $A^\star$ given by $\dps A^\star_h p = A_0 p + \int_Q \polarization_p^h$ for any $p \in \RR^d$.
\end{theorem}
Note that the definition of $A^\star_h$ is directly inspired from~\eqref{eq:tau_Astar}.

\begin{proof}
The proof essentially goes by building an appropriate $r_h(\tau_p) \in \approximations_0$ which is a good approximation of $\tau_p$ in $L^2(\unitcell_0)$. For any $\beta \in \I$, let $x^h_\beta$ denote the center of the cell $\cell{\beta}$. We build $r_h(\tau_p)$ as follows:
\begin{itemize}
\item[(i)] For any $\beta \not\in \Izero$, i.e. such that $\cell{\beta} \not\subset \unitcellzero$, we set $r_h(\tau_p)(x) = 0$ on $\cell{\beta}$.
\item[(ii)] For any $\beta \in \Izero$ such that the cell $\cell{\beta}$ is accross two or more domains $D_m$, we set $r_h(\tau_p) = 0$ on $\cell{\beta}$.
\item[(iii)] For any $\beta \in \Izero$ such that the cell $\cell{\beta}$ is a subset of one of the sets $D_m$, we set $r_h(\tau_p) = \tau_p(x^h_\beta)$ on $\cell{\beta}$.
\end{itemize}
The function $r_h(\tau_p)$ obviously belongs to $\approximations_0$.

We now estimate $\| \tau_p - r_h(\tau_p) \|_{L^2(\unitcell)}$, splitting the sum over $\beta$ according to the three cases above:
\begin{eqnarray*}
\| \tau_p - r_h(\tau_p) \|^2_{L^2(\unitcell)}
&=&
\sum_{\beta \in \I} \int_{\cell{\beta}} \left| \tau_p(x) - r_h(\tau_p)(x) \right|^2 \, dx
\\
&=&
\sum_{\beta \in (i)} \int_{\cell{\beta}} \left| \tau_p(x) \right|^2 \, dx
+
\sum_{\beta \in (ii)} \int_{\cell{\beta}} \left| \tau_p(x) \right|^2 \, dx
+
\sum_{\beta \in (iii)} \int_{\cell{\beta}} \left| \tau_p(x) - \tau_p(x_\beta^h) \right|^2 \, dx.
\end{eqnarray*}
For the cells in the group (i), if $\cell{\beta} \subset \unitcell \setminus \unitcellzero$, we have $\tau_p(x) = 0$ on $\cell{\beta}$. In the first sum, we are thus left with cells that are accross the boundary of $\unitcell_0$. For the cells in the group (iii), we note that $\tau_p$ is $C^\delta$ on the cell, hence $\left| \tau_p(x) - \tau_p(x_\beta^h) \right| \leq C \left| x - x_\beta^h \right|^\delta$. We thus have
\begin{equation*}
\| \tau_p - r_h(\tau_p) \|^2_{L^2(\unitcell)}
\leq
\sum_{\text{cells accross $\partial \unitcell_0$}} \int_{\cell{\beta}} \left| \tau_p(x) \right|^2 \, dx
+
\sum_m \sum_{\text{cells accross $\partial D_m$}} \int_{\cell{\beta}} \left| \tau_p(x) \right|^2 \, dx
+
C \sum_{\beta \in (iii)} \int_{\cell{\beta}} \left| x - x_\beta^h \right|^{2\delta} \, dx.
\end{equation*}
Since $\tau_p \in C^\delta(\overline{D_m})$ for all $m$, we see that $\tau_p \in L^\infty(\unitcell)$. In addition, since the boundaries of $\unitcell_0$ and of $D_m$ are regular, the number of cells $\cell{\beta}$ that are accross these boundaries scales as $C h^{1-d}$ for some $C$ independent of $h$. We therefore deduce that
\begin{equation*}
\| \tau_p - r_h(\tau_p) \|^2_{L^2(\unitcell)}
\leq
C h \| \tau_p \|_{L^\infty(\unitcell)}
+
C h^{2 \delta}
\leq 
C h^{\min(1, 2\delta)}. 
\end{equation*}
Using~\eqref{eq:cea}, we now obtain
\begin{equation*}
\| \tau_p - \tau^h_p \|_{L^2(\unitcellzero)} 
\leq 
C \inf_{\test^h \in \approximations_0} \| \tau_p - \test^h \|_{L^2(\unitcellzero)}
\leq 
C \| \tau_p - r_h(\tau_p) \|_{L^2(\unitcellzero)}
\leq 
C \| \tau_p - r_h(\tau_p) \|_{L^2(\unitcell)}
\leq 
C h^{\min(1/2, \delta)},
\end{equation*}
which is the claimed bound~\eqref{eq:error}. The bound on the homogenized coefficient directly follows. This concludes the proof of Theorem~\ref{theo:error}.
\end{proof}

Note that, if the grid is consistent with the heterogeneities of $\Aper$ (i.e. if any cell $\cell{\beta}$ is a subset of some $D_m$) and with the set $\unitcell_0$, then the above proof yields that $\| \tau_p - \tau^h_p \|_{L^2(\unitcellzero)} \leq C h^\delta$. If, in addition, $\tau_p$ were very smooth on $\overline{\unitcell}$, the best result we can hope for is $\| \tau_p - \tau^h_p \|_{L^2(\unitcellzero)} \leq C h$, since we use a piecewise approximation of $\tau_p$. In that case, we would have $\left| A^\star - A^\star_h \right| \leq C h$.

\section{Numerical results}
\label{sec:num}

We now use the numerical method described above to estimate the homogenized properties of the 3D microstructure shown on Fig.~\ref{fig:1}, where 500 spherical inclusions are randomly located in the unit-cell $\unitcell$ by means of a standard hard-spheres Monte-Carlo simulation. All inclusions share the same radius $r \approx 0.0575$. The total volume fraction of inclusions is $40\,\%$. The minimal distance between two inclusions is close to the diameter $2r$ of the inclusions. The materials in-between inclusions is called the matrix.

\begin{figure}[htbp]
  \centering
  \includegraphics{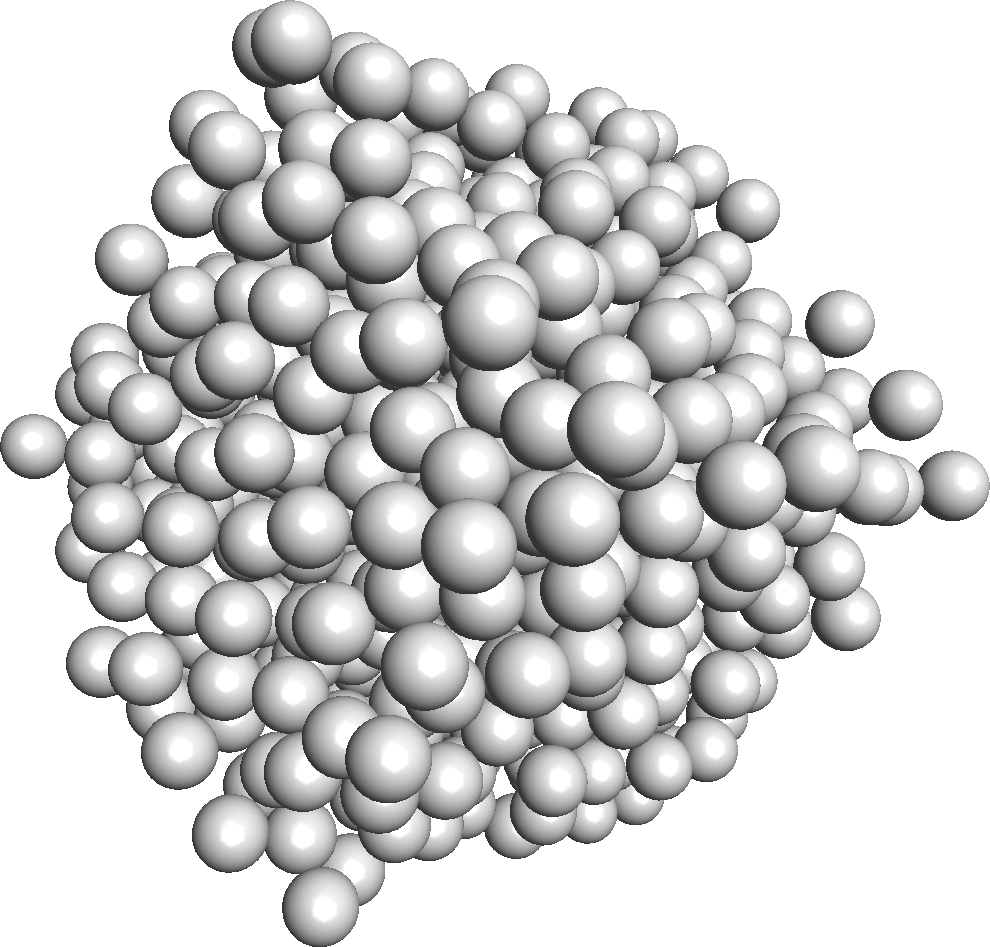}
  \caption{The 3D microstructure considered for the numerical test.}
  \label{fig:1}
\end{figure}

In the above sections, we have assumed that the problem of interest, namely~\eqref{eq:pb_eps_intro}, is a {\em scalar} problem: $u^\eps$ is a real-valued field, and $\Aper$ is a matrix, i.e. a second-order tensor. However, the approach carries over to other types of linear elliptic problems. The simulation presented here is carried out within the framework of linear elasticity. The problem then again reads as in~\eqref{eq:pb_eps_intro}, where $u^\eps$ is now a vector-valued field (representing the displacement), and $\Aper$ is a fourth-order tensor. The problem has a variational structure, hence the tensor $\Aper$ satisfies a symmetry-like property. In the corrector problem~\eqref{PB:correcteur_intro}, $p$ is a second-order symmetric tensor, which is interpreted as a macroscopic strain. The fourth-order homogenized stiffness tensor is then computed by solving six corrector problems (with six linearly independent values of $p$). We focus here on one of these six corrector problems, by choosing for $p$ a unit macroscopic shear strain along the $xy$ plane:
\begin{equation}
\label{eq:choix_p}
p_{\text{shear}} = \left( \begin{array}{ccc} 0 & 1 & 0 \\ 1 & 0 & 0 \\ 0 & 0 & 0 \end{array} \right).
\end{equation}
We decided to work with a shear strain rather than with a uniaxial strain, of the form
\begin{equation*}
p = \left( \begin{array}{ccc} 1 & 0 & 0 \\ 0 & 0 & 0 \\ 0 & 0 & 0 \end{array} \right),
\end{equation*}
because computations under the former are generally more difficult than under the latter. More iterations are typically required to reach convergence.

The matrix and the inclusions are modelled as linear elastic, isotropic materials, the mechanical parameters of which are given in Table~\ref{tab:mech}. In order to show the robustness of the method, we have chosen to work with a high contrast between the matrix and inclusion mechanical properties (the inclusions are 1000 times stiffer than the matrix).

\medskip

The corrector problem~\eqref{PB:correcteur_intro} (with $p_{\text{shear}}$ given by~\eqref{eq:choix_p}) is discretized following the approach described in Section~\ref{sec:disc}. To estimate the equivalent voxel stiffness $A^h_\beta$ defined by~\eqref{eq:def_ah} for different grid sizes $h$, we proceed as follow. Rather than computing the right-hand side of~\eqref{eq:def_ah} for perfectly spherical inclusions, we actually slightly modify the definition of the reference microstructure. This reference microstructure is defined on a very fine grid (of size $h_{\rm ref}=1/4096$) as the following binary microstructure: each voxel is given the value 1 if its center lies within a spherical inclusion; otherwise, the voxel value is zero. The reference inclusions are then the union of the voxels with value 1. Their shape is close to be a ball, but not exactly (up to an error controlled by $h_{\rm ref}$). 

We discretize~\eqref{eq:5} using grids of size $h$, with $1/h=8$, $16$, $32$, $64$, $128$, $256$, $512$ and $1024$. The number of voxels in the unit cell $\unitcell$ is therefore $8^3$, $16^3$, $32^3$, $64^3$, $128^3$, $256^3$, $512^3$ and $1024^3$, respectively. On each voxel $\beta$, $\tau^h_\beta$ is a constant $3 \times 3$ symmetric matrix. The number of degrees of freedom in~\eqref{eq:13} for the finest grid is therefore $6.44 \, \times \, 10^9$. Since the reference microstructure is defined on a fine grid, computing~\eqref{eq:def_ah} is actually straightforward, and amounts to performing local averages. Numerical quadrature rules are not needed.

The reference medium $A_0$ is modelled as a linear elastic, isotropic material, the mechanical parameters of which are given in Table~\ref{tab:mech}. In practice, we do not use the exact, consistent discrete Green operator $\hat{\green}_0^{h,\text{cons}}$, which is too difficult to evaluate, but its filtered approximation described in Section~\ref{sec:filtered}. Since the reference medium is softer than the matrix and the inclusions, the matrix of the linear system~\eqref{eq:13} (with $\hat{\green}_0^{h,\text{cons}}$ replaced by $\hat{\green}_0^{h,\text{filt}}$) is positive definite (see~\cite{BRIS2010A}). We use a conjugate gradient solver to solve~\eqref{eq:13}.

\begin{table}[htbp]
  \centering
  \begin{tabular}{c|c|c|c}
    & matrix & inclusions & reference medium $A_0$ \\
    \hline
    shear modulus & $\mu_{\rm m}$ & $\mu_{\rm i} = 1000 \, \mu_{\rm m}$ & $\mu_0 = 0.5 \, \mu_{\rm m}$ \\ 
    Poisson ratio & $\nu_{\rm m} = 0.3$ & $\nu_{\rm i} = 0.2$ & $\nu_0=0.3$ \\
    \hline
  \end{tabular}
  \caption{Mechanical parameters of the matrix, inclusions and reference medium.}
  \label{tab:mech}
\end{table}

\medskip

The numerical results are presented in Table~\ref{tab:2}, where $A^\star_{xyxy} := p_{\text{shear}} : A^\star : p_{\text{shear}}$. We observe that the difference between the values of $A^\star_{xyxy}$ computed on the two finest grids is roughly of $2\,\%$. We also note that, as expected (see~\cite{BRIS2010A}), the estimation of $A^\star_{xyxy}$ increases as $N$ increases (see also Fig.~\ref{fig:2}).

In addition, if we define as reference value the value of $A^\star_{xyxy}$ obtained on the $1024^3$ grid, then it seems that the numerical error is proportional to $h$. If Theorem~\ref{theo:error} also holds in the case of linear elasticity (recall it is stated in the case of a scalar PDE, in particular because the work~\cite{li-vogelius} is stated for scalar PDEs), then it would ensure that the numerical error is bounded by $C h^\delta$ for some $\delta<1$. The better rate that we observe in practice may be related to the fact that $\Aper$ has a better regularity than just being piecewise H\"older continuous: it is actually piecewise constant. See Remark~\ref{rem:regu} above in that spirit. 

\begin{table}[htbp]
  \centering
  \begin{tabular}{c|c|c|c}
    $N=1/h$ & $A^\star_{xyxy}$ & Wall-clock time [sec] & \# iterations $\mathcal{N}$ \\
    \hline
    8 & 2.3221 & $8.47 \ 10^{-3}$ & 13\\
    16 & 2.4771 & $3.26 \ 10^{-2}$ & 19\\
    32 & 2.8765 & $2.99 \ 10^{-1}$ & 32\\
    64 & 3.4226 & 4.18 & 55\\
    128 & 3.8595 & 58.5 & 91\\
    256 & 4.1096 & 861 & 154\\
    512 & 4.2319 & $1.02 \ 10^4$ & 227\\
    1024 & 4.3284 & $1.0 \ 10^5$ & 269\\
    \hline
  \end{tabular}
  \caption{Estimate of $A^\star_{xyxy}$ as a function of the size of the grid, $N=1/h$. For each grid, we also show the total (wall-clock) time spent in the solver, and the number of iterations $\mathcal{N}$ in the conjugate gradient algorithm to reach convergence (up to a relative error of $10^{-5}$).}
  \label{tab:2}
\end{table}

\begin{figure}[htbp]
  \centering
  \includegraphics{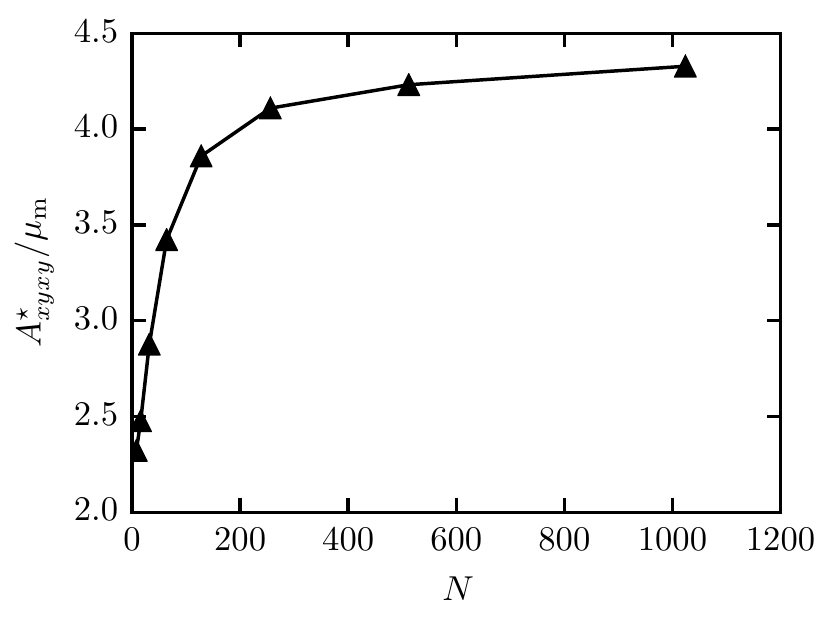}
  \caption{Estimate of $A^\star_{xyxy}$ (normalized by $\mu_{\rm m}$) as a function of the grid-size $N=1/h$.}
  \label{fig:2}
\end{figure}

The cost of each iteration of the linear solver is dominated by the cost of the two FFTs, which scales as $N^3\log(N)$. The total computational time $T$ spent in the solver should therefore scale as $\mathcal{N} \, N^3\log(N)$, where $\mathcal{N}$ is the number of iterations of the conjugate gradient algorithm. On Fig.~\ref{fig:3}, we show the ratio $T / \Big( \mathcal{N} \, N^3\log(N) \Big)$ as a function of $N$, and indeed observe that, for the largest values of $N$, this ratio is roughly a constant.

\begin{figure}[htbp]
  \centering
  \includegraphics{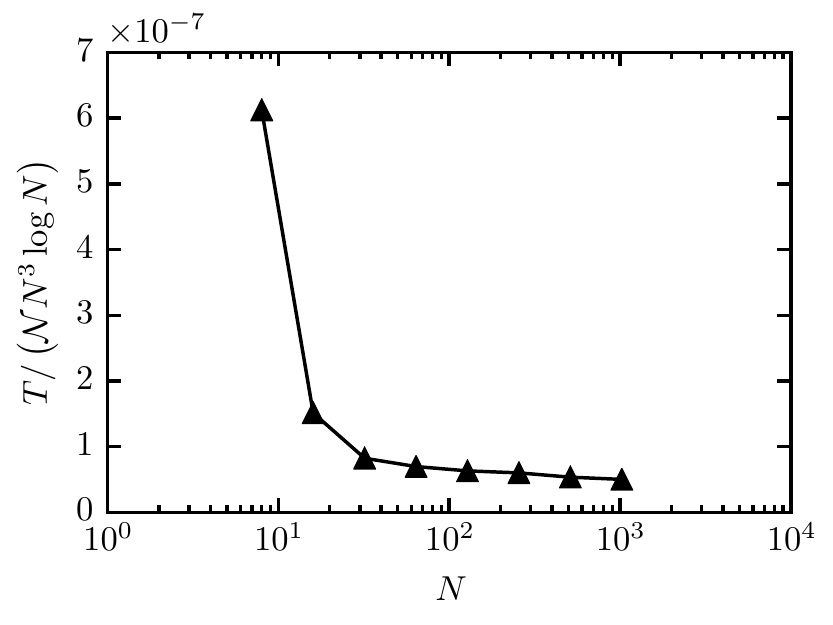}
  \caption{Ratio $T / \Big( \mathcal{N} \, N^3\log(N) \Big)$ as a function of the grid size $N=1/h$.}
  \label{fig:3}
\end{figure}

We eventually investigate the efficiency of the parallelization described in Section~\ref{sec:implem}. We consider the grid with $128^3$ voxels, and solve the associated linear system using $P=1, 2, 4, 8$ processes. Let $T_P$ denote the time spent in the solver when $P$ processes are used. The efficiency index is defined as
\begin{equation}
  \label{eq:17}
  \mathcal E=\frac{T_1}{P \, T_P}.
\end{equation}
A perfect parallelization corresponds to observing $\mathcal E = 1$. As shown on Table~\ref{tab:3}, our implementation is very efficient for 2 and 4 processes, where $\mathcal E$ is close to 1. It is less the case when using $P=8$ processes. Similar trends are observed on the $256^3$ grid. This probably means that, when using $P=8$ processes or more, too much time is spent in communications rather than in local computations. This loss of efficiency will be investigated in the future.

\begin{table}[htbp]
  \centering
  \begin{tabular}{c|c|c}
    $P$ & $T_P$ [sec] & $\mathcal E$\\
    \hline
    1 & 270 & -- \\
    2 & 151 & 0.9 \\
    4 & 76 & 0.9 \\
    8 & 58 & 0.6 \\
    \hline
  \end{tabular}
  \caption{Computation time $T_P$ and efficiency index $\mathcal E$ defined by~\eqref{eq:17} when varying the number $P$ of processes. All the results given in the above tables and figures are obtained with $P=8$ processes.}
  \label{tab:3}
\end{table}

\section*{Acknowledgments}

We are indebted to William Minvielle for a careful proof-reading of a previous version of that manuscript. FL would like to thank Jed Brown and Matthew G. Knepley, who organized a mini-symposium within the PMAA 2014 conference, for their kind invitation. 

\bibliographystyle{elsarticle-num}
\bibliography{manuscript_brisard_legoll}

\end{document}